\documentclass{amsart}
\usepackage{amssymb,amsmath,amsrefs}
\usepackage[colorlinks=true]{hyperref}
\hypersetup{urlcolor=blue, citecolor=green}
\usepackage[dvips]{graphicx}
\usepackage{color}
\usepackage{subcaption}
\usepackage{float}

\graphicspath{{Imagens/}}

\theoremstyle{plain}

\newtheorem*{conj*}{Conjecture}
\newtheorem*{cor*}{Corollary}
\newtheorem{theorem}{Theorem}[section]

\newtheorem{proposition}[theorem]{Proposition}
\newtheorem{corollary}[theorem]{Corollary}
\newtheorem{lemma}[theorem]{Lemma}

\newtheorem{question}{Question}

\theoremstyle{definition}
\newtheorem*{def*}{Definition}

\newtheorem{definition}[theorem]{Definition}

\newcommand{\SE}{{\mathcal E}}

\newcommand{\N}{\mathbb{N}}

\newcommand{\eps}{\varepsilon}

\newcommand{\mc}[1]{\mathcal{{#1}}}

\newcommand{\tn}{\textnormal}

\newcommand{\dist}{\operatorname{\textit{d}}}

\newcommand{\inte}{\operatorname{Int}}

\newcommand{\espinha}{\operatorname{Spin}}

\newcommand{\diam}{\operatorname{diam}}

\title{Continuum-wise hyperbolic homeomorphisms on surfaces}
\author{Rodrigo Arruda}
\author{Bernardo Carvalho}
\author{Alberto Sarmiento}
\date{\today}

\begin{document}

\begin{abstract}
This paper discusses the dynamics of continuum-wise hyperbolic surface homeomorphisms. We prove that $cw_F$-hyperbolic surface homeomorphisms containing only a finite set of spines are $cw_2$-hyperbolic. In the case of $cw_3$-hyperbolic homeomorphisms we prove the finiteness of spines and, hence, that $cw_3$-hyperbolicity implies $cw_2$-hyperbolicity. In the proof, we adapt techniques of Hiraide \cite{Hiraide1987ExpansiveHO} and Lewowicz \cite{Lewowicz1989ExpansiveHO} in the case of expansive surface homeomorphisms to prove that local stable/unstable continua of $cw_F$-hyperbolic homeomorphisms are continuous arcs. We also adapt techniques of Artigue, Pac\'ifico and Vieitez \cite{Artigue2013NexpansiveHO} in the case of N-expansive surface homeomorphisms to prove that the existence of spines is strongly related to the existence of bi-asymptotic sectors and conclude that spines are necessarily isolated from other spines.

\end{abstract}

\maketitle

\section{Introduction and statement of results}

In the classification of expansive homeomorphisms on surfaces, Lewowicz and Hiraide \cites{Lewowicz1989ExpansiveHO,Hiraide1987ExpansiveHO} prove that expansiveness implies the homeomorphism is pseudo-Anosov. The hypothesis in the dynamics is only expansiveness, the space is a surface, and from these a complete understanding of the structure of local stable/unstable sets is obtained: they are subsets of $C^0$ transversal singular foliations with a finite number of singularities.
A classification of topologically hyperbolic homeomorphisms on surfaces was also obtained \cite{hiraide2toro}. Expansive surface homeomorphisms satisfying the shadowing property are conjugate to an Anosov diffeomorphism of the Torus $\mathbb{T}^2$. Expansiveness and shadowing together define a local product structure that is continuous and rule out the presence of singularities on local stable/unstable sets.

With the study of dynamics beyond topological hyperbolicity (see \cites{antunes2023firsttime,Artigue2019CountablyAE,ARTIGUE20203057,artigue2020continuumwise,CARVALHO20163734,carvalho2019positively}), which is being developed as the understanding of generalizations of hyperbolicity from a topological perspective, we could try to extend the reach of these techniques to distinct and more general scenarios. This idea was first explored by Artigue, Pac\'ifico and Vieitez \cite{Artigue2013NexpansiveHO} in the study of N-expansive homeomorphisms on surfaces. They exhibit an example of a 2-expansive homeomorphism on the bitorus that is not expansive, and prove that non-wandering 2-expansive homeomorphisms on surfaces are, indeed, expansive. The techniques of Lewowicz/Hiraide are applied to study the structure of bi-assymptotic sectors, which are discs bounded by a local stable and a local unstable arc. Bi-asymptotic sectors naturally appear in the dynamics of 2-expansive homeomorphisms, and it is proved in \cite{Artigue2013NexpansiveHO} that these sectors should be contained in the wandering part of the system. The techniques seem to be restricted to the case of 2-expansive homeomorphisms and extend them to 3-expansiveness seems complicated. In particular, use them to understand the dynamics of more general continuum-wise expansive homeomorphisms, introduced by Kato in \cite{kato_1993}, seems difficult. 

However, in the study of cw-expansive homeomorphisms a recent work of Artigue, Carvalho, Cordeiro and Vieitez \cite{artigue2020continuumwise} discussed the continuum-wise hyperbolicity, assuming that local stable and unstable continua of sufficiently close points of the space intersect (see Definition \ref{cwf}). Cw-hyperbolic systems share several important properties with the topologically hyperbolic ones, such as the L-shadowing property \cite{ARTIGUE20203057} and a spectral decomposition theorem \cite{artigue2020continuumwise}, but a few important differences are noted on the pseudo-Anosov diffeomorphism of $\mathbb{S}^2$: the existence of stable/unstable spines, bi-asymptotic sectors, cantor sets in arbitrarily small dynamical balls, and a cantor set of distinct arcs in local stable/unstable sets.

In this paper we start to adapt the techniques of Lewowicz/Hiraide to the study of cw-hyperbolic homeomorphisms on surfaces. One important hypothesis in our results is cw$_{F}$-expansiveness, that asks for a finite number of intersections between any pair of local stable and local unstable continua (see Definition \ref{cwf}). We do not know an example of a cw-hyperbolic surface homeomorphism that is not cw$_{F}$-expansive, so the study of cw$_{F}$-hyperbolicity on surfaces seems to be the perfect first step in the theory. In our first result, we prove that local stable/unstable continua of cw$_{F}$-hyperbolic surface homeomorphisms are arcs (see Proposition \ref{arcs}). Cw$_{F}$-hyperbolicity is important in this step since in \cite{artigueanomalous} there is a cw-expansive surface homeomorphism with non locally connected local stable continua. But even in the case they are locally connected, they are only assured to be contained in dendritations as proved in \cite{artigue_2018}. Section 2 is devoted to prove that they are arcs and this is done in a few important steps that are based in the ideas of Lewowicz and Hiraide.

In our second result we relate bi-asymptotic sectors and spines in a way that every regular sector contains a single spine and every spine is contained in a regular bi-asymptotic sector. The notion of regular bi-asymptotic sector is defined and pictures of non-regular sectors are presented. We give a complete description of the structure of local stable and local unstable continua inside a regular bi-asymptotic sector. We also prove that every bi-asymptotic sector contains a regular bi-asymptotic sector and, hence, a spine, and that non-regular sectors contain at least two distinct spines. All these results on sectors and spines are proved in Section 3 and allow us to conclude that the spines of a cw$_{F}$-hyperbolic surface homeomorphism are isolated from other spines and, hence, we obtain that there is at most a countable number of them.
%and in \cite{AAV} there is a cw-expansive homeomorphism with an infinite number of fixed points.

In Section 4 we prove our main result using all the techniques developed in the previous sections. The hypothesis of the existence of at most a finite number of spines is important and will be discussed. We do not know examples of cw-hyperbolic surface homeomorphisms with an infinite number of spines, so this hypothesis also seems reasonable. The following is the main result of this article:

\begin{theorem}\label{main}
If a cw$_{F}$-hyperbolic surface homeomorphism has only a finite number of spines, then it is cw2-hyperbolic. Cw3-hyperbolic surface homeomorphisms have at most a finite number of spines, and are cw2-hyperbolic.
\end{theorem}

We note that in \cite{artigue2020continuumwise} it is proved that the product of $n$ copies of the pseudo-Anosov diffeomorphism of $\mathbb{S}^2$ is cw$_{2^n}$-hyperbolic but is not cw$_{2^n-1}$-expansive. Thus, the hypothesis on the space being a surface is important for our main result. We state the following questions that follow naturally from our results:

\begin{question}
        Does there exist a $cw$-hyperbolic surface homeomorphism that is not $cw_2$-hyperbolic?
    \end{question}

    \begin{question}
        Does $cw_F$-hyperbolicity on surfaces imply finiteness on the number of spines and, hence, $cw_2$-hyperbolicity? 
    \end{question}

    \begin{question}
        Does $cw$-hyperbolicity on surfaces imply local connectedness of $C^s_\eps(x)$?
    \end{question}

\begin{question}
        Can we adapt the techniques of this paper to prove that 3-expansive surface homeomorphisms are 2-expansive?
    \end{question}
    
    \vspace{+0.4cm}

\section{Local stable/unstable continua are arcs}
    
We begin this section with some precise definitions. Let $(X,d)$ be a compact metric space and $f\colon X\to X$ be a homeomorphism.
We consider the \emph{c-stable set} of $x\in X$ as the set 
$$W^s_{c}(x):=\{y\in X; \,\, d(f^k(y),f^k(x))\leq c \,\,\,\, \textrm{for every} \,\,\,\, k\geq 0\}$$
and the \emph{c-unstable set} of $x$ as the set 
$$W^u_{c}(x):=\{y\in X; \,\, d(f^k(y),f^k(x))\leq c \,\,\,\, \textrm{for every} \,\,\,\, k\leq 0\}.$$
We consider the \emph{stable set} of $x\in X$ as the set 
$$W^s(x):=\{y\in X; \,\, d(f^k(y),f^k(x))\to0 \,\,\,\, \textrm{when} \,\,\,\, k\to\infty\}$$
and the \emph{unstable set} of $x$ as the set 
$$W^u(x):=\{y\in X; \,\, d(f^k(y),f^k(x))\to0 \,\,\,\, \textrm{when} \,\,\,\, k\to-\infty\}.$$
%The dynamical ball of $x$ with radius $c$ is the set $$\Gamma_{c}(x)=W^u_{c}(x)\cap W^s_{c}(x).$$ 
%We say that $f$ is \emph{expansive} if there exists $c>0$ such that $$\Gamma_c(x)=\{x\} \,\,\,\,\,\, \text{for every} \,\,\,\,\,\, x\in X.$$ We say that $f$ is \emph{continuum-wise expansive} if there exists $c>0$ such that $\Gamma_{c}(x)$ is totally disconnected for every $x\in X$.
We denote by $C^s_c(x)$ the $c$-stable continuum of $x$, that is the connected component of $x$ on $W^s_{c}(x)$, and denote by $C^u_c(x)$ the $c$-unstable continuum of $x$, that is the connected component of $x$ on $W^u_{c}(x)$. 
    
    \begin{definition}\label{cwf}
        We say that $f$ is $cw$-expansive if there exists $c>0$ such that $$W^s_c(x)\cap W^u_c(x) \quad \text{ is totally disconnected }$$ for every $x\in X$.
        A $cw$-expansive homeomorphism is said to be $cw_F$-expansive if there exists $c>0$ such that $$\#(C^s_c(x)\cap C^u_c(x))<\infty \quad \text{ for every } \quad x\in X.$$
        Analogously, $f$ is said to be $cw_N$-expansive if there is $c>0$ such that $$\#(C^s_c(x)\cap C^u_c(x))\leq N \quad \text{ for every } \quad x\in X.$$    
        We say that $f$ satisfies the $cw$-local-product-structure if for each $\eps>0$ there exists $\delta>0$ such that $$C^s_\eps(x)\cap C^u_\eps(y)\neq \emptyset \quad \text{ whenever } \quad \dist(x,y)< \delta.$$ The $cw$-expansive homeomorphisms (resp.\ $cw_F$, $cw_N$) satisfying the $cw$-local-product-structure are called $cw$-hyperbolic (resp.\ $cw_F$, $cw_N$).
    \end{definition}

The main examples of cw-hyperbolic surface homeomorphism are the Anosov diffeomorphisms (or more generally the topologically hyperbolic homeomorphisms), and the pseudo-Anosov diffeomorphism of $\mathbb{S}^2$. The sphere $\mathbb{S}^2$ can be seen as the quotient of $\mathbb{T}^2$ by the antipodal map, and thus any 2x2 hyperbolic matrix $A$ with integer coefficients and determinant one induces diffeomorphisms $f_A$ on $\mathbb{T}^2$ and $g_A$ on $\mathbb{S}^2$. The diffeomorphism $f_A$ is Anosov and, hence, cw1-hyperbolic, while $g_A$ is cw2-hyperbolic but not cw1-expansive (see \cites{Artigue2019CountablyAE,artigue2020continuumwise,ARTIGUE20203057} for more details). 
We recall some known properties of local stable/unstable continua for cw$_F$-hyperbolic homeomorphisms. Most of them hold assuming that $X$ is a Peano continuum, that is a compact, connected, and locally connected metric space, but we assume from now on that $f\colon S\to S$ is a cw$_F$-hyperbolic homeomorphism of a closed surface $S$. We will use the symbol $\sigma$ to denote both $s$ and $u$. Since this will appear several times in what follows, we will not write in all of them that some statement holds for $\sigma=s$ and $\sigma=u$, we will simply say that it holds for $\sigma$.

    \begin{lemma}[\cite{katoconcerning} Thm. 1.6]\label{tamanho uniforme}
        For every $\eps>0$, there exists $\delta>0$ such that $$\diam(C^{\sigma}_\eps(x))>\delta \,\,\,\,\,\, \text{for every} \,\,\,\,\,\, x\in S,$$ where $\diam(A)=\sup\{d(a,b); a,b\in A\}$ denotes the diameter of the set $A$.
    \end{lemma}

    \begin{corollary}\label{interior}
        $\inte C^\sigma_\eps(x) = \emptyset$ for every $x\in S$ and $\eps\in(0,\frac{c}{2})$.
    \end{corollary}
    \begin{proof}
        The following is the proof for $\sigma=s$, but for $\sigma=u$ the proof is analogous.
        By contradiction, assume that $y\in \inte C^s_\eps(x)\neq \emptyset$ for some $x\in X$ and $\eps<\frac{c}{2}$. Since $\diam C^u_\eps(y)> \delta$ for some $\delta>0$, then $C^s_\eps(x)\cap C^u_\eps(y)$ contains a non-trivial continuum. By the choice of $\eps$, we have that $C^u_\eps(y)\subset C^u_c(y)$, then $C^s_c(x)\cap C^u_c(x)$ contains a non-trivial continuum, which contradicts the $cw$-expansiveness.
    \end{proof}
    
    Let $\mathcal{C}$ denote the space of all sub-continua of $S$ and $\mathcal{C}_\delta$ denote the set of all sub-continua of $S$ with diameter smaller than $\delta$. Let $\mathcal{C}^s$ and $\mathcal{C}^u$ denote the set of all stable and unstable continua of $f$, respectively. More precisely,
    $$\mathcal{C}^s=\{C\in\mathcal{C}; \,\,\, \diam(f^n(C))\to0, \,\,\, n\to\infty\}$$
    $$\mathcal{C}^u=\{C\in\mathcal{C}; \,\,\, \diam(f^n(C))\to0, \,\,\, n\to-\infty\}.$$ The following lemma is actually a characterization of cw-expansiveness with respect to local stable/unstable continua.
    \begin{lemma}[\cite{artigue_2018} Prop. 2.3.1]\label{cwexpansivecarac}
    \,
        \begin{enumerate}
            \item There exists $\eps^*>0$ such that $C^\sigma_{\eps^*}\subset C^\sigma$,
            \item For all $\eps>0$ there exists $\delta>0$ such that $C^\sigma\cap C_\delta \subset C^\sigma_\eps$.
        \end{enumerate}
    \end{lemma}
    The following lemma is similar to Lemma \textcolor{red}{4.1} in \cite{Hiraide1987ExpansiveHO} but assuming cw-expansiveness. Let $B_{\delta}(x)$ denote the ball of radius $\delta$ centered at $x$, that is the set of points whose distance to $x$ is less than $\delta$. 
    \begin{lemma}\label{epsigual2eps}
        For each $0<\eps<\frac{\eps^*}{4}$ there exists $\delta\in(0,\eps)$ such that $$C^\sigma_\eps(x)\cap B_\delta(x)= C^\sigma_{2\eps}(x)\cap B_\delta(x)$$ for every $x\in S$.
    \end{lemma}
    \begin{proof}
        For each $0<\eps<\frac{\eps^*}{4}$ let $\delta^*\in(0,\eps)$ be given by Lemma \ref{cwexpansivecarac} and $\delta=\frac{\delta^*}{2}$. Since $\eps<\frac{\eps^*}{4}$, it follows that $$C^\sigma_\eps(x)\subset C^\sigma_{2\eps}(x)\in C^\sigma.$$ The choice of $\delta$ ensures that $C^\sigma_{2\eps}(x)\cap B_\delta(x)\subset C^\sigma_\eps(x)$, and, hence, $$C^\sigma_{2\eps}(x)\cap B_\delta(x) = C^\sigma_\eps(x)\cap B_\delta(x).$$
    \end{proof}
%Hereafter, let $S$ be a compact surface and $f:S\to S$ a $cw_F$-homeomorphism.
    In the following lemma, the hypothesis of cw$_F$-expansiveness will be necessary. It was first proved in \cite{artigue_2018} for cw$_F$-expansive homeomorphisms, but we include a different proof using cw$_F$-hyperbolicity. In this article, an arc is a subset of $S$ homeomorphic to $[0,1]$.
    
    \begin{lemma}[\cite{artigue_2018} Thm. 6.7.1]\label{conexidadelocal}
    There exists $\eps>0$ such that $C^\sigma_\eps(x)$ is locally connected for every $x\in S$.
    \end{lemma}
    
    \begin{proof}
        We prove for $\sigma=s$ but the proof for $\sigma=u$ is similar. Let $0<\eps<\frac{c}{4}$ and by contradiction assume that $C^s_{\eps}(x)$ is not locally connected for some $x\in S$. Then we can consider a sequence of arcs $(P_n)_{n\in\N}\subset C^s_{\eps}(x)$ (as in the proof of Proposition 3.1 in \cite{Hiraide1987ExpansiveHO}) such that $P_i\cap P_j=\emptyset$ if $i\neq j$ and $\dist(P_n, P^*)\to 0$ for some non-trivial arc $P^*\subset C^s_\eps(x)$ (here we also denote by $d$ the Hausdorff distance on the space of continua). Let $y\in P^*$ be an interior point of $P^*$ and consider $r\in(0,\eps)$ such that \begin{enumerate}
            \item $P^*$ separates $B_r(y)$, and
            \item $P_n$ separates $B_r(y)$ for every $n>n_0$ and for some $n_0\in\N$.
        \end{enumerate}
        We can assume that every $P_n$ is contained in the same component $A$ of $B_r(y)\setminus P^*$ taking a sub-sequence if necessary. The cw-local-product-structure ensures the existence of $\delta\in(0,\frac{r}{4})$ such that
        $$C^s_{\frac{r}{4}}(a)\cap C^u_{\frac{r}{4}}(b)\neq \emptyset \quad \text{ whenever } \quad \dist(a,b)< \delta.$$ Since $\inte C^s_{2\eps}(y)=\emptyset$, there exists $z\in A\cap B_{\delta}(y)$ such that $z\notin C^s_{2\eps}(y)$, and, hence, $$C^s_r(z)\cap C^s_{\eps}(y)=\emptyset.$$ In particular, $C^s_r(z)\cap P_n=\emptyset$ for every $n\in\N$. Choose $n_1>n_0$ such that $P_{n_1}\cap B_r(y)$ separates $P^*\cap B_r(y)$ and $C^s_{\frac{r}{4}}(z)$. The choice of $\delta$ ensures that $C^u_{\frac{r}{4}}(y)\cap C^s_{\frac{r}{4}}(z)\neq\emptyset$. In particular, $C^u_{\frac{r}{4}}(y)\cap P_{n_1}\neq\emptyset$, and, hence, $C^u_{\eps}(y)$ intersects an infinite number of distinct $P_n's$ (see Figure \ref{fig:locconexo}). This contradicts cw$_F$-expansiveness and finishes the proof.
        \begin{figure}[H]
            \centering
            \includegraphics[width=0.15\textwidth]{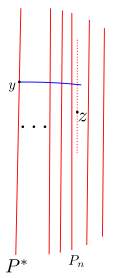}
            \caption{}
            \label{fig:locconexo}
        \end{figure}
    \end{proof}
    
    The following corollary is a consequence of this result. A subset $A\subset S$ is arcwise connected if for each pair of distinct points $x,y\in A$ there exists an arc $h\colon[0,1]\to A$ such that $h(0)=x$ and $h(1)=y$.
    
    \begin{corollary}\label{arccon}
        There exists $\eps>0$ such that $C^\sigma_\eps(x)$ is arcwise connected and locally arcwise connected for every $x\in S$. Moreover, for each pair of distinct points $y,z\in C^\sigma_\eps(x)$, there is a unique arc $\sigma(y,z;x)$ in $C^\sigma_\eps(x)$ connecting $y$ and $z$.
    \end{corollary}
    \begin{proof}
        From Lemma \ref{conexidadelocal} and Theorem 5.9 of \cite{hall1955elementary}, it follows that $C^\sigma_\eps(x)$ is a Peano space. Hence, Theorem 6.29 of \cite{hall1955elementary} ensures that $C^\sigma_\eps(x)$ is arcwise connected and locally arcwise connected. The uniqueness comes from the observation that two distinct arcs connecting $y$ and $z$ would create an open set bounded by a local stable (in the case $\sigma=s$) or local unstable (in the case $\sigma=u$) curve, which contradicts cw-expansiveness on surfaces.
    \end{proof}
    From now on we choose $\eps>0$ given by Corollary \ref{arccon} and $\delta\in(0,\eps)$ satisfying the Lemmas \ref{tamanho uniforme}, \ref{cwexpansivecarac}, and \ref{epsigual2eps}. Following the steps of Hiraide in \cite{Hiraide1987ExpansiveHO}, we define an equivalence relation in the set of arcs starting on $x$ and contained in $C^{\sigma}_{\eps}(x)$. 
    
    \begin{definition}
    Let $x\in S$, and $y,z\in C^{\sigma}_{\eps}(x)$. We write $y\sim z$ if \[
     \sigma(x,y;x)\cap \sigma(x,z;x) \supsetneq \{x\}.
    \] We define the number of stable/unstable separatrices at $x$ as $$P^\sigma(x)= \#(C^\sigma_\eps(x)/\sim).$$ Lemma \ref{epsigual2eps} ensures that the number of separatrices at $x$ does not depend on the choice of $\eps<\frac{\eps^*}{4}$. This explains the notation $P^\sigma(x)$ without $\eps$ being mentioned.
    \end{definition}
    
    The following lemma follows an idea present in the works of Lewowicz/Hiraide: stable and unstable separatrices must be alternated as in Figure \ref{fig:alternado}. In this step, the cw$_F$-hyperbolicity will also be necessary. Let $\partial B_{\delta}(x)$ denote the boundary of the ball $B_{\delta}(x)$, that is the set of points whose distance to $x$ equals $\delta$.
    
    \begin{figure}[H]
        \centering
        \includegraphics[width=0.3\textwidth]{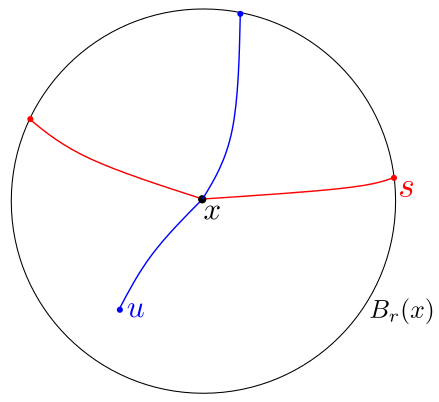}
        \caption{Lemma \ref{alternado}}
        \label{fig:alternado}
    \end{figure}
    
    \begin{lemma}\label{alternado}
        For each $x\in S$, there exists $r_0>0$ such that if $r\in(0,r_0)$, $y,z\in\partial B_{r}(x)$ are in different classes of $\sim$, $s(x,y;x)$ and $s(x,z;x)$ are arcs intersecting $\partial B_{r}(x)$ at one point and $A$ is a component of $B_{r}(x)\setminus (s(x,y;x)\cup s(x,z;x))$, then there is $a\in C^{u}_{\eps}(x)\cap A$ such that $u(x,a;x)\subset A$. 
        %Exchanging $s$ and $u$, we have the same result. 
    \end{lemma}
    
    \begin{proof}
        If $C^{s}_{\eps}(x)\cap C^{u}_{\eps}(x)=\{x\}$, let $r_0=\eps$, otherwise, $$C^{s}_{\eps}(x)\cap C^{u}_{\eps}(x)\supsetneq \{x\},$$ and we let $$r_0=\dist(x,(C^{s}_{\eps}(x)\cap C^{u}_{\eps}(x))\setminus\{x\}),$$ which is a positive number by $cw_F$-expansiveness.
        Let $x$, $r$, $y$, $z$, and $A$ as above, and suppose there is no unstable arc $u(x,a;x)\subset A$. Choose $\delta_r\in(0,\frac{r}{4})$, given by the  cw-local-product-structure, such that
        $$C_\frac{r}{4}(x)\cap C_\frac{r}{4}(y)\neq\emptyset \,\,\, \tn{  whenever  } \,\,\, d(x,y)<\delta_r.$$
        Lemma \ref{interior} assures the existence of 
        $$b\in A\setminus C^{s}_{\eps}(x) \,\,\,\,\,\, \text{with} \,\,\,\,\,\, \dist(b,x)<\frac{\delta_r}{2}.$$ It follows that $$C^{s}_{\frac{r}{4}}(b)\subset A.$$
        If $C^{s}_{\eps}(x)\cap C^{u}_{\eps}(x)=\{x\}$, then $C^{u}_{\frac{r}{4}}(x)\cap A=\emptyset$ since there is no unstable arc $u(x,a;x)\subset A$; otherwise, $C^{s}_{r}(x)\cap C^{u}_{r}(x)=\{x\}$ since $r<r_0$, and, hence, $C^{u}_{\frac{r}{4}}(x)\cap A=\emptyset$.

    \begin{figure}[H]
        \centering
        \includegraphics[width=0.3\textwidth]{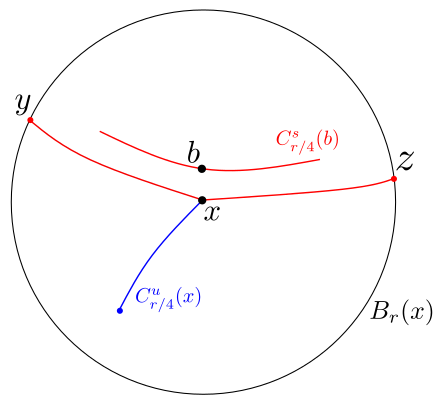}
        \caption{}
        \label{fig:alternado0}
    \end{figure}

        In both cases, we obtain $$C^{s}_{\frac{r}{4}}(b)\cap C^{u}_{\frac{r}{4}}(x)=\emptyset.$$
        This contradicts the choice of $\delta_r$ since $\dist(b,x)<\frac{\delta_r}{2}$ (see Figure \ref{fig:alternado0})
        and proves the existence of an unstable arc $u(x,a;x)\subset A$ and finishes the proof.
    \end{proof}

    The following corollary is a direct consequence of the previous Lemma. %We write $\sigma^{-1}=u$ if $\sigma=s$, and $\sigma^{-1}=s$ if $\sigma=u$.

    \begin{corollary}\label{separatrices}
        $P^s(x)=P^u(x)$ for every $x\in S$. 
        %If $P^\sigma(x)<\infty$, then $P^s_\eps(x)=P^u_\eps(x)$. If $P^\sigma(x)=\infty$, then $P^{\sigma^{-1}}(x)=\infty$.
    \end{corollary}
    
    We recall that $C^{\sigma}_{\eps}(x)$ is locally connected but the case $P^\sigma(x)=\infty$ is still not ruled out since we could have an infinite number of separatrices with diameter converging to zero. Also, in the expansive case, examples of pseudo-Anosov homeomorphisms with singularities containing a number of stable/unstable separatrices greater than two can be constructed. In the following lemma, we observe that these two scenarios do not occur in the case of cw$_F$-hyperbolic homeomorphisms. Indeed, the existence of bifurcation points contradicts cw$_F$-hyperbolicity.

    \begin{lemma}\label{1or2}
        $P^{\sigma}(x)\leq2$ for every $x\in S$.
    \end{lemma}
    \begin{proof}
        Suppose, by contradiction, that there exists $x\in S$ with $P^{\sigma}(x)\geq 3$.
        Let $r>0$ and $x_1, x_2, x_3$ be in different classes of $C_\eps^s(x)$ such that $s(x,x_i;x) \subset B_r(x)$ and intersects $\partial B_r(x)$ only at $x_i$ for every $i=1,2,3$. Then $\bigcup_{i=1}^3 s(x,x_i;x)$ separates $B_r(x)$ in exactly three components, $A_1, A_2, A_3$ as in Figure \ref{fig:arco1}.
        
    \begin{figure}[H]
        \centering
        \includegraphics[width=0.3\textwidth]{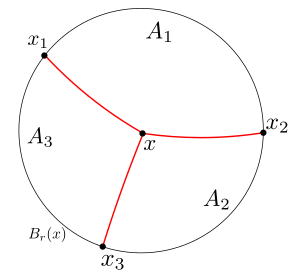}
        \caption{}
        \label{fig:arco1}
    \end{figure}
Using Lemma \ref{alternado}, we can find $y_i\in A_i$ such that $u(x,y_i;x) \subset A_i$. 
If $C^s_{\eps}(x)\cap C^u_\eps(x)=\{x\}$, let $$r_0 = \min(\diam u(x,y_i;x))$$ otherwise, let $$r_0=\min(\dist(x,(C^{s}_{\eps}(x)\cap C^{u}_{\eps}(x))\setminus\{x\}), \min(\diam u(x,y_i;x))),$$ which is a positive number by $cw_F$-expansiveness.
        It follows that $\bigcup_{i=1}^3 u(x,z_i;x)$ divides $B_{r_0}(x)$ in three components, $B_1,B_2,B_3$, for some $z_i\in u(x,y_i;x)$ (see Figure \ref{fig:arco2}). 

    \begin{figure}[H]
        \centering
        \includegraphics[width=0.3\textwidth]{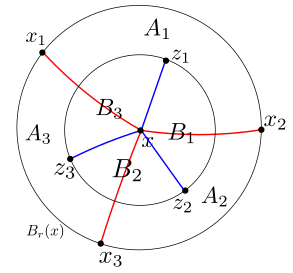}
        \caption{}
        \label{fig:arco2}
    \end{figure}
        Since the stable and unstable arcs are alternating, then $$A_i\cap B_j=\emptyset.$$ for some $i,j\in \{1,2,3\}$ (see Figure \ref{fig:arco3}).

    \begin{figure}[H]
        \centering
        \includegraphics[width=0.3\textwidth]{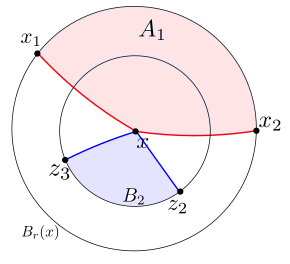}
        \caption{}
        \label{fig:arco3}
    \end{figure}
Choose $\delta_{r_0}\in (0,\frac{r_0}{4})$, given by cw-local-product-structure, such that $$C^s_\frac{r_0}{4}(x)\cap C^u_\frac{r_0}{4}(y)\neq\emptyset \,\,\, \tn{  whenever  } \,\,\, d(x,y)<\delta_{r_0}.$$
    Lemma \ref{interior} assures the existence of $$a\in A_i\setminus C^s_\eps(x) \,\,\,\,\,\, \text{with} \,\,\,\,\,\, \dist(a,x)<\frac{\delta_{r_0}}{2}$$ and $$b\in B_j\setminus C^u_\eps(x) \,\,\,\,\,\, \text{with} \,\,\,\,\,\, \dist(b,x)<\frac{\delta_{r_0}}{2}.$$      
    It follows that $$C^s_{\frac{r_0}{4}}(a)\subset A_i \,\,\,\,\, \text{and} \,\,\,\,\, C^u_{\frac{r_0}{4}}(b)\subset B_j,$$ and, hence, $C^s_{\frac{r_0}{4}}(a)\cap C^u_{\frac{r_0}{4}}(b)=\emptyset$ (see Figure \ref{fig:arco4}). 
    \begin{figure}[H]
        \centering
        \includegraphics[width=0.3\textwidth]{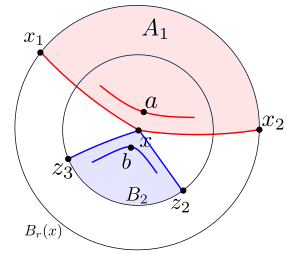}
        \caption{}
        \label{fig:arco4}
    \end{figure}    
        This contradicts the choice of $\delta_{r_0}$ and finishes the proof.
    \end{proof}
    
    There are two possible cases: either $P^{\sigma}(x)=1$ and $x$ is said to be a spine, or $P^{\sigma}(x)=2$ and $x$ is said to be a regular point. Let $\espinha(f)$ denote the set of all spines of $f$. The following proposition gathers all results we obtained so far. We prove that $C^\sigma_\eps(x)$ is an arc for every $x\in S$. 
    
    \begin{proposition}\label{arcs}
        If $x\in\espinha(f)$, then there is a homeomorphism $h^{\sigma}:[0,1]\to C^\sigma_\eps(x)$ with $h^{\sigma}(0)=x$. Otherwise, there is a homeomorphism $h^{\sigma}:[-1,1]\to C^\sigma_\eps(x)$ with $h^{\sigma}(0)=x$.
    \end{proposition}
    
    \begin{proof}
    %\textcolor{red}{Estrutura da prova: explain what an end point is; explain why $C^\sigma_\eps(x)$ has an end point no caso da espinha and two end points no caso regular; dizer que $C^\sigma_\eps(x)$ fica homeomorfo a [0,1] ou [1,1].}
    Let $IC^{\sigma}(x)$ denote the union of all open arcs in $C^{\sigma}_{\eps}(x)$ and $$BC^{\sigma}(x)=C^{\sigma}_{\eps}(x)\setminus IC^{\sigma}(x).$$ Note that $BC^{\sigma}(x)$ is always formed by two distinct points $x_1$ and $x_2$, since local connectedness of $C^{\sigma}_{\eps}(x)$ ensures the existence of at least two points in $BC^{\sigma}(x)$ (as in Lemma 4.5 of \cite{Hiraide1987ExpansiveHO}), and the existence of three distinct points in $BC^{\sigma}(x)$ would imply the existence of $y\in C^{\sigma}_{\eps}(x)$ with $P^{\sigma}(y)\geq3$, contradicting Lemma \ref{1or2}. If $x\in \espinha(f)$, then either $x=x_1$ or $x=x_2$, and the arc connecting $x_1$ to $x_2$ gives us a homeomorphism $h^{\sigma}\colon [0,1]\to C^{\sigma}_{\eps}(x)$ such that $h^{\sigma}(0)=x$. If $x\notin\espinha(f)$, then $x\notin BC^{\sigma}(x)$ and the arc connecting $x_1$ to $x_2$ gives us a homeomorphism $h^{\sigma}:[-1,1]\to C^\sigma_\eps(x)$ with $h^{\sigma}(0)=x$.
    %By Lemmas \ref{arccon} and \ref{1or2}, we have that $C^\sigma_\eps(x)$ is homeomorphic to either $[0,1)$ or $[0,1]$ via some homeomorphism $h^\sigma$, with $h^\sigma(x)=0$. To complete the proof, we show that $C^\sigma_\eps(x)$ cannot be homeomorphic to $[0,1)$. Indeed, suppose we had a homeomorphism $h^\sigma: [0,1)\to C^\sigma_\eps(x)$, then $x$ would be the unique endpoint of $C^\sigma_\eps(x)$. As $C^\sigma_\eps(x)\subset B_\eps(x)$, by compactness we would have a sequence of subarcs $P_n\subset C^\sigma_\eps(x)$ and a nontrivial continuum $P\subset C^\sigma_\eps(x)$ such that $P_i\cap P_j=\emptyset$ for $i\neq j$, and $\dist(P_n, P)\to 0$. This contradicts the fact that $C^\sigma_\eps(x)$ is locally connected and proves that $C^\sigma_\eps(x)$ is homeomorphic to $[0,1]$.
%It remains to prove for $x\in S\setminus \espinha(f)$. In the same way as above, Lemmas \ref{arccon} and \ref{1or2} ensures that $C^\sigma_\eps(x)$ is homeomorphic to either $[-1,1]$, $(-1,1]$, $[-1,1)$, or $(-1,1)$ via some homeomorphism $h^\sigma$, with $h^\sigma(x)=0$. Similar to the previous case, if $C^\sigma_\eps(x)$ is homeomorphic to $(-1,1]$, $[-1,1)$, or $(-1,1)$, we would obtain the same contradiction with the local connectedness of $C^\sigma_\eps(x)$. Therefore, we conclude that $C^\sigma_\eps(x)$ is homeomorphic to $[-1,1]$.
    \end{proof}
    
    Now we exhibit two important consequences of above results that will be important in the proofs of Section 3. In the first lemma, we prove that either a local stable/unstable continuum separates a small ball, or it contains a spine in this ball. The notation $c.c_x (A)$ is used to denote the connected component of $x$ in the set $A$. %shows that for every $x\in S$ either the stable/unstable continua cut the boundary of small neighborhoods $V$ of $x$ twice or it contains a spine contained in that neighborhood, i.e.\ , if $C^\sigma_\eps(x)$ does not contain a spine, then each class of $C^\sigma_\eps(x)/\sim$ has an element $y$ such that $[x,y]_\sigma \cap \partial B_\delta(x)\neq \emptyset$.

    \begin{lemma}\label{separadelta}
        For each $0<\eps<\frac{\eps^*}{4}$, there exists $\delta\in(0,\eps)$ such that for each $x\in S$ one of the following holds:
    \begin{enumerate}
        \item $c.c_x (C^\sigma_\eps(x)\cap B_\delta(x))$ separates $B_\delta(x)$,
        \item $c.c_x (C^\sigma_\eps(x)\cap B_\delta(x))$ contains a spine.
    \end{enumerate}
    \end{lemma}
    \begin{proof}
    If $x\in\espinha$, we are in case (2). Then we assume that     $x\in S\setminus \espinha(f)$. For each $0<\eps<\frac{\eps^*}{4}$, let $\delta\in(0,\eps)$ be given by Lemma \ref{epsigual2eps} such that $$C^\sigma_\eps(x)\cap B_\delta(x) = C^\sigma_{2\eps}(x)\cap B_\delta(x).$$
Let $h^{\sigma}:[-1,1]\to C^\sigma_\eps(x)$ be a homeomorphism as in Proposition \ref{arcs} with $h^{\sigma}(0)=x$. Lemma \ref{tamanho uniforme} ensures the existence of $z\in C^\sigma_\eps(x)\cap \partial B_\delta(x)$. Without loss of generality, we can assume that $z\sim h(-1)$. We assume item (1) is false and prove item (2). Suppose that $c.c_x (C^\sigma_\eps(x)\cap B_\delta(x))$ does not separate $B_\delta(x)$. Since $z\in C^\sigma_\eps(x)\cap \partial B_\delta(x)$ and $z\sim h(-1)$, then $\sigma(x,y;x) \subset \inte B_\delta(x)$, where $y=h(1)$. Since $y\in C^\sigma_\eps(x)$, then $C^\sigma_\eps(y)\subset C^\sigma_{2\eps}(x)$ and, hence, $$C^\sigma_\eps(y)\cap B_{\delta}(x)\subset C^\sigma_{2\eps}(x)\cap B_{\delta}(x)=C^\sigma_\eps(x)\cap B_\delta(x).$$ Thus, $y\in c.c_x (C^\sigma_\eps(x)\cap B_\delta(x))$ and $P(y)=1$, that is, $y\in\espinha(f)$.

%By Lemma \ref{epsigual2eps} we have that $C^\sigma_\eps(x)\cap B_\delta(x)=C^\sigma_{2\eps}(x)\cap B_\delta(x)$, then the connected component of $x$ in $C^\sigma_{2\eps}(x)\cap B_\delta(x)$ coincides with that $x$ in $C^\sigma_\eps(x)\cap B_\delta(x)$. Since $C^\sigma_\eps(y)\cap B_{\delta}(x)\subset C^\sigma_{2\eps}(x)\cap B_{\delta}(x)$, we have that $P(y)=1$, because $y$ is an endpoint of $C^\sigma_\eps(x)$. Therefore $C^\sigma_\eps(x)$ contains a spine.
    \end{proof}
    
    In the last result of this section, we observe that local stable/unstable sets intersect transversely. First, we state a precise definition for topological transversality.
    
    \begin{definition}
        Let $\alpha, \beta$ be arcs in $S$ meeting at $x$. We say that $\alpha$ is topologically transversal to $\beta$ at $x$ if there exists a disk $D$ such that
        \begin{enumerate}
            \item $\alpha \cap \beta \cap D=\{x\}$,
            \item $\beta$ separates $D$, and
            \item the connected components of $(\alpha\setminus \beta)\cap D$ are in different components of $D\setminus \beta$ (See Fig.\ \ref{fig:transversearcs}).
        \end{enumerate} 
    \end{definition}

    \begin{figure}[H]
        \centering
        \includegraphics[width=0.23\textwidth]{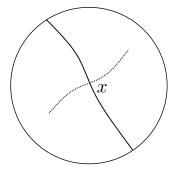}
        \caption{Two arcs topologically transverse at the point $x$.}
        \label{fig:transversearcs}
    \end{figure}

    \begin{lemma}\label{transversal}
        If $x,y\in S$ and $z\in C^s_\eps(x)\cap C^u_\eps(y)\setminus (Spin(f))$, then $C^s_\eps(x)$ intersects $C^u_\eps(y)$ transversely at $z$.
    \end{lemma}
    \begin{proof}
        Since $z\notin Spin(f)$, then $P^s(z)=P^u(z)=2$. If the intersection is not transversal we find $a_1,a_2\in C^s_{2\eps}(z)$ and a disk $D$ around $z$ such that $s(a_1,a_2;z)$ separates $D$ and there is a component of $D\setminus s(a_1,a_2;z)$ containing $b_1\not\sim b_2\in C^u_\eps(z)$. This is a contradiction with \ref{alternado} because $P_\eps(z)=2$.
    \end{proof}

     \vspace{+0.4cm}

    \section{Bi-asymptotic sectors and spines}
    
    Bi-asymptotic sectors were introduced in \cite{Artigue2013NexpansiveHO} for $N$-expansive homeomorphisms on surfaces. These sectors were defined as being a disk bounded by the union of a local stable and a local unstable arc (see Figure \ref{fig:setorantigo}). In the case of 2-expansive homeomorphisms, a consequence of the arguments of \cite{Artigue2013NexpansiveHO} is that both intersections $a_1,a_2$ of a bi-asymptotic sector are not only transversal, but point outside the disk (see Figure \ref{fig:setorantigo}). Indeed, 2-expansiveness and non-existence of wandering points imply the non-existence of spines inside by-asymptotic sectors (see Proposition 3.5 in \cite{Artigue2013NexpansiveHO}), and this ensure the existence of a third intersection between the stable and unstable arcs bounding the sector if the intersection points inward.

    \begin{figure}[H]
        \begin{subfigure}[h]{0.49\linewidth}
            \centering
            \includegraphics[width=0.76\textwidth]{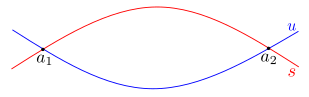}
            \caption{}
            \label{fig:setorantigo}
        \end{subfigure}
        \hfill        
        \begin{subfigure}[h]{0.49\linewidth}
            \centering
            \includegraphics[width=0.76\textwidth]{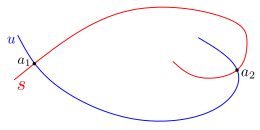}
            \caption{}
            \label{fig:setorbugado1}
        \end{subfigure}%
        \caption{}
    \end{figure}

    For cw$_F$-hyperbolic homeomorphisms, it is not possible to ensure the intersections points outward the disk. First, because there will be spines inside the sector, but also because this case allows more intersections between local stable/unstable arcs.
    The goal is to prove that every bi-asymptotic sector contains a spine and that every spine is contained in a bi-asymptotic sector. To prove this, we first understand the case of bi-asymptotic sectors with intersections pointing outward the sector (these sectors will be called regular). We will characterize the structure of stable/unstable arcs inside a regular bi-asymptotic sector obtaining a single spine inside it. 
    
    \begin{definition}\label{bidefinition}
    We say that $C^s_\eps(x)$ and $C^u_\eps(x)$ form a bi-asymptotic sector if there exists a pair of sub-arcs $a^s$, $a^u$ contained in $C^s_\eps(x)$ and $C^u_\eps(x)$, respectively, such that $a^s\cup a^u$ bounds a disk $D$. In this case, $D$ is called the bi-asymptotic sector. Let $a_1$ and $a_2$ be the end points of $a^s$ and $a^u$.
    A bi-asymptotic sector $D$ is said to be regular if it satisfies the following:
    \vspace{+0.2cm}
    
        (Regularity condition) There exists neighborhoods $V_{a_1}$ of $a_1$ and $V_{a_2}$ of $a_2$ such that $C^{\sigma}_\eps(x)\cap V_{a_1}\cap\inte D=\emptyset$ and $C^{\sigma}_\eps(x)\cap V_{a_2}\cap\inte D=\emptyset$.
    \end{definition}

Without the regularity condition, the bi-asymptotic sectors can contain more than one spine, and, hence, a more complicated structure of stable/unstable arcs, as in Figure \ref{fig:setorbugado21}, \ref{fig:setorbugado22} and \ref{fig:setorbugado33}.
    
    \begin{figure}[H]
        \begin{subfigure}[h]{0.4\linewidth}
            \centering
            \includegraphics[width=0.7\textwidth]{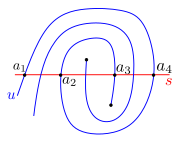}
            \caption{}
            \label{fig:setorbugado21}
        \end{subfigure}
        \hfill    
        \begin{subfigure}[h]{0.4\linewidth}
            \centering
            \includegraphics[width=0.7\textwidth]{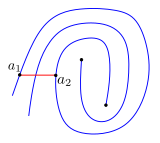}
            \caption{}
            \label{fig:setorbugado22}
        \end{subfigure}
        \begin{subfigure}[h]{0.4\linewidth}
            \centering
            \includegraphics[width=0.7\textwidth]{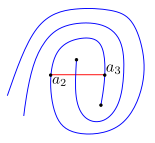}
            \caption{}
            \label{fig:setorbugado33}
        \end{subfigure}
        \hfill
        \begin{subfigure}[h]{0.4\linewidth}
            \centering
            \includegraphics[width=0.7\textwidth]{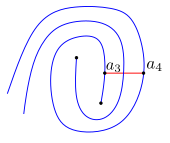}
            \caption{}
            \label{fig:setorbugado44}
        \end{subfigure}
        \caption{}
    \end{figure}   

    Note that both the stable and unstable continua enter the disk passing through $a_2$, that is, for every neighborhood $V$ of $a_2$ we have $C^\sigma_\eps(a_2)\cap V\cap\inte D\neq \emptyset$. Thus, the sector in Figure \ref{fig:setorbugado22} formed considering the stable arc from $a_1$ and $a_2$ does not satisfy the regularity condition. The same happens with the sector bounded by the stable and unstable arcs connecting $a_3$ and $a_4$. Also, note that the sectors formed by the stable arc from $a_1$ to $a_4$ and from $a_2$ to $a_3$ satisfy the regularity condition. Inside these sectors, the structure of stable/unstable arcs is the same: there is a single spine and all stable/unstable arcs turn around this spine. We prove in this section that this is exactly the structure of stable/unstable arcs inside any regular bi-asymptotic sector. 
    Let $D$ be a regular bi-asymptotic sector bounded by $a^s$ and $a^u$ with $\diam D<\delta$ (given by Lemma \ref{separadelta}), and let $a_1$ and $a_2$ be the end points of $a^s$ and $a^u$. For $p\in D$, define $C^u_D(p)$ and $C^s_D(p)$ as the connected component of $C^u(p)\cap D$ and $C^s(p)\cap D$ containing $p$ respectively. We remark that Lemma \ref{separadelta} also holds changing the ball $B_{\delta}(x)$ for the sector $D$, that is, for each $p\in D$, either $C^{\sigma}_D(p)$ separates $D$ or $C^{\sigma}_D(p)$ contains a spine. The hypothesis of regularity is important to ensure the following result.

    \begin{lemma}\label{as}
    $C^{\sigma}_D(p)=a^{\sigma}$ for every $p\in a^{\sigma}$.
    \end{lemma}
    \begin{proof}
        It is clear that $C^{\sigma}_D(p)\supset a^{\sigma}$. 
        By contradiction, assume that there exist $p\in a^{\sigma}$ and $y\in C^{\sigma}_D(p)\setminus a^{\sigma}$. This means that either $\sigma(a_1,y;x)$ or $\sigma(a_2,y;x)$ is contained in $\inte D$ and contradicts the regularity condition.
    \end{proof}
    Note that in Figure \ref{fig:setorbugado22}, the arc $a^u$ connects $a_1$ to $a_2$, while $C^u_D(a_1)$ contains $a^u$ and also an arc from $a_2$ to a spine in the interior of $D$. Also, $a^s$ is the stable arc from $a_1$ to $a_2$, while $C^s_D(a_1)$ contains the arc connecting $a_2$ to $a_4$ (see Figure \ref{fig:setorbugado21}). The following lemma is a consequence of the previous lemma and the transversality explained in Lemma \ref{transversal}. The following results also hold changing the roles of $s$ and $u$ but we will not use $\sigma$ as before since it would make the presentation more complicated.

    \begin{lemma}\label{setorbonito}
        $1\leq \#(C^s_D(p)\cap a^u)\leq 2$ for every $p\in D$. If $\#(C^s_D(p)\cap a^u)=1$, then $C^s_D(p)$ contains a spine.
    \end{lemma}
    \begin{proof}
        %Since $\diam D<\delta$, then $C^s_D(p)$ coincides with the connected component of $C^s_\eps(p)\cap D$ containing $p$ by Lemma \ref{cwexpansivecarac}.
        If $p\in a^s$, then $C^s_D(p)=a^s$ since $D$ is regular (see Lemma \ref{as}). This implies that $\#(C^s_D(p)\cap a^u) = \#(a^s\cap a^u)=2$ (See Fig. \ref{fig:setorantigo}).
        If $p\in D\setminus a^s$, then $C^s_D(p)\cap a^s=\emptyset$, by Lemma \ref{1or2}, which implies that $C^s_D(p)\cap a^u \neq \emptyset$ since Lemma \ref{tamanho uniforme} ensures that $C^s_D(p)\cap (a^s\cup a^u) \neq \emptyset$. Therefore $\#(C^s_D(p)\cap a^u)\geq 1$.
        Note that every intersection between $C^s_D(p)$ and $a^u$ is transversal and this together with Lemma \ref{1or2} ensure that $\#(C^s_D(p)\cap a^u)\leq 2$.
        The last part of the statement is obtained using Lemma \ref{separadelta}, since $\#(C^s_D(p)\cap a^u)=1$ implies that $C^s_D(p)$ does not separate $D$, and, hence, contains a spine.
    \end{proof}
    Note that in the non-regular sectors of Figure \ref{fig:setorbugado21}, while $a^s\cap a^u=\{a_1,a_2\}$, we have $C^s_D(p)\cap a^u=\{a_1,a_2,a_3\}$. 
    %However, for $p\in D\setminus a^s$, transversality and Lemma \ref{1or2} are sufficient to ensure that $#(C^s_D(p)\cap a^u)\leq 2$ as in the proof above. 
    Following \cite{Artigue2013NexpansiveHO}, we define an order in the set $\mc{F}^s=\{C^s_D(x):x\in D\}$ as follows: $C^s_D(x)<C^s_D(y)$ if $a^s$ and $C^s_D(y)$ are separated by $C^s_D(x)$, i.e.\ , $a^s$ and $C^s_D(y)$ are in different components of $D\setminus C^s_D(x)$. Note that $a^s$ is a minimal element for the order and if $y\in \inte D\cap\espinha$, then $C^s_D(y)$ does not separate $D$ and $C^s_D(y)$ cannot be smaller than $C^s_D(z)$ for any $z\neq y$ in $D$.
    The following lemma is based on Lemma 3.2 in \cite{Artigue2013NexpansiveHO}. The regularity condition on the sector allows us to basically follow the original proof.
    
    \begin{lemma}\label{order}
        The order $<$ in $\mc{F}^s$ is total.
    \end{lemma}
    \begin{proof}
        Let $C^s_D(x)$ and $C^s_D(y)$ be different elements of $\mc{F}^s$ and suppose by contradiction that neither $C^s_D(x)<C^s_D(y)$ nor $C^s_D(y)<C^s_D(x)$. We assume that $x$ and $y$ are spines since in the other cases we obtain the result similarly. Consider $\gamma_1,\gamma_2,\gamma_3\subset a^u$ sub-arcs as in Figure \ref{fig:ot1}.
        \begin{figure}[H]
        \centering
        \includegraphics[width=1\textwidth]{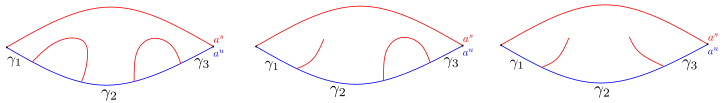}
        \caption{Non-comparable stable arcs.}
        \label{fig:ot1}
    \end{figure}
        Since $x$ and $y$ are spines, it follows that $E = D\setminus (C^s_D(x)\cup C^s_D(y))$ is connected. For $1\leq i < j \leq 3$, define $$A_{ij}=\{x\in E:C^s_D(x)\cap \gamma_i\neq \emptyset, C^s_D(x)\cap \gamma_j \neq \emptyset\}.$$
        By the definition of the subarcs, we have that $A_{ij}$ is non-empty for all $1\leq i < j \leq 3$. In addition, these sets are closed and cover $E$. Since $E$ is connected, we can find $z$ that belongs to all of them. Hence $\#(C^s_D(z)\cap a^u) \geq 3$ and this contradicts Lemma \ref{setorbonito}. In the other cases we just need to choose the appropriate arcs $\gamma_i$ and change the definition of the set $E$ accordingly. Figure \ref{fig:order4} illustrates these choices.
        %and by Lemma \ref{setorbonito} we have that $\#(C^s_D(p)\cap a^u) \leq 2$ for all $p\in D$.
        \begin{figure}[H]
            \centering
            \includegraphics[width=0.50\textwidth]{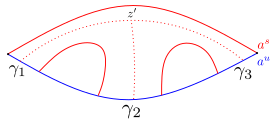}
            \caption{}
            \label{fig:order4}
        \end{figure}
    \end{proof}
    Note that the regularity condition is important to conclude the order is total since non-regular sectors can contain points $z\in\inte D$ such that $\#(C^s_D(z)\cap a^u)\geq 3$, so the existence of a point in the intersection of $A_{ij}$ would not imply a contradiction.
    
    Lemma \ref{order} ensures that inside a regular bi-asymptotic sector there is at most one spine, but does not necessarily prove the existence of a spine. This will be consequence of the following result that proves continuity for the variation of the arcs inside a regular bi-asymptotic sector. It is based on Lemma 3.3 in \cite{Artigue2013NexpansiveHO}.
    Lemma \ref{setorbonito} ensures that $\#(C^s_D(x)\cap a^u)\leq2$
    for every $x\in a^u$. Then we consider a map $g:a^u\to a^u$ (see Figure \ref{fig:deforder}) defined as $$C^s_D(x)\cap a^u=\{x,g(x)\}.$$ Note that if $C^s_D(x)\cap a^u=\{x\}$, then $g(x)=x$ and Lemma \ref{setorbonito} ensures that $C^s_D(x)$ contains a spine.
        \begin{figure}[H]
            \centering
            \includegraphics[width=0.50\textwidth]{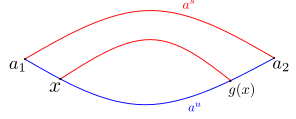}
            \caption{}
            \label{fig:deforder}
        \end{figure}
    Note that we could have problems to define $g$ in sectors that are not regular since in these cases $C^s_D(a_1)$ may not coincide with $a^s$ and intersect $a^u$ in three different points.

    \begin{lemma}
        The map $g:a^u\to a^u$ is continuous.
    \end{lemma}
    \begin{proof}
        As $a^u$ is homeomorphic to $[0,1]$, we can induce an ordering in $a^u$ such that $a_1<a_2$. We prove that $g$ is decreasing with this order, and since $g$ is bijective, we conclude continuity.
        Suppose, by contradiction, that $g$ is not decreasing, so there exist $x<y$ such that $g(x)<g(y)$. Note that $g(x)\neq x$ since $x<y$ and $x=g(x)<g(y)$ ensure the arcs $s(x,g(x);x)$ and $s(y,g(y);y)$ are not comparable, contradicting Lemma \ref{order}. The same reason ensures that $g(y)\neq y$. If $x<y<g(x)<g(y)$, then there is an intersection between $s(x,g(x);x)$ and $s(y,g(y);y)$, contradicting Lemma \ref{1or2}. If $x<g(x)<y<g(y)$ the arcs $s(x,g(x);x)$ and $s(y,g(y);y)$ are not comparable, contradicting Lemma \ref{order}. Other cases are obtained from these cases interchanging $x$ and $g(x)$ or $y$ and $g(y)$, leading to the same contradictions.
    \end{proof}

    The following proposition gathers all results we obtained so far about bi-asymptotic sectors. It is one of the directions of the equivalence between the existence of sectors and spines that we want to prove.

    \begin{proposition}\label{espinhasetor}
        If $D$ is a bi-asymptotic sector with diameter less than $\delta$, then $\inte D\cap \espinha(f)\neq \emptyset$. Moreover, if $D$ is regular, then $\# (\inte D\cap \espinha(f))=1$.
    \end{proposition}
    \begin{proof}
        First, we note that every regular bi-asymptotic sector $D$ contains a unique spine in its interior. Since $a^u$ is homeomorphic to $[0,1]$ and $g\colon a^u\to a^u$ is continuous, it follows that $g$ has a fixed point, and since $g$ is decreasing, this fixed point is unique. Clearly, $a_1$ and $a_2$ are not fixed points of $g$. This and the transversality proved in Lemma \ref{transversal} ensure that this single spine is contained in $\inte D$.
        
        Now let $D$ be a non-regular bi-asymptotic sector. 
        %To establish that $\inte D\cap \espinha(f)\neq \emptyset$, we will show that either $D$ contains a spine in its interior or $D$ contains a regular bi-asymptotic sector. This will complete the proof. 
        Without loss of generality, let us assume that $D$ does not satisfy the regularity condition at $a_2$, that is, $C^s_\eps(x)$ enters $D$ through $a_2$. If $C^s_\eps(x)$ does not intersect the open arc $a^u\setminus\{a_1,a_2\}$, then the connected component of $C^s_\eps(x)$ containing $a_2$ does not separate $D$, and by Lemma \ref{separadelta} we find a spine in $\inte D$. If $C^s_\eps(x)$ intersects $a^u\setminus \{a_1,a_2\}$ at a point $y$, then the transversality of the intersection ensures that $$s(a_2,y;x)\cup u(a_2,y;x)$$ bounds a regular bi-asymptotic sector contained in $D$ (see Figure \ref{fig:setorcontido}). Thus, there exists a spine in $\inte D$.
        \begin{figure}[H]
            \centering
            \includegraphics[width=0.4\textwidth]{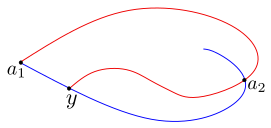}
            \caption{}
            \label{fig:setorcontido}
        \end{figure}
    \end{proof}

    Now we prove the other direction of the equivalence. 

    \begin{lemma}\label{spinesector}
    If $x\in \espinha(f)$, then there exists $D_x$ a regular bi-asymptotic sector such that $x\in \inte D_x$.
    \end{lemma}
    \begin{proof}
        We begin choosing $y\in C^{s}_{\eps}(x)\cap B_{\frac{\delta}{2}}(x)$ such that $s(x,y;x) \subset B_{\frac{\delta}{2}}(x)$. Since $C^{u}_{\frac{\eps}{2}}(y)$ is an arc transversal to $C^{s}_{\eps}(x)$ at $y$ and $y\notin\espinha(f)$, we can choose a small neighborhood $U$ of $C^{s}_{\eps}(x)$ and $t_1,t_2\in C^u_{\frac{\eps}{2}}(y)\cap \partial U$ satisfying:
        \begin{enumerate}
            \item $t_1\not \sim t_2$,
            \item $u(t_1,t_2;y)$ intersects $\partial U$ only at $t_1$ and $t_2$,
            \item $u(t_1,t_2;y)\cap C^s_\eps(x)=\{y\}$.
        \end{enumerate}
        Since $u(t_1,t_2;y)$ is transversal to $C^s_\eps(x)$ at $y$, and $u(t_1,t_2;y)\cap C^s_\eps(x)=\{y\}$, then $u(t_1,t_2;y)$ divides both $U$ and $C^s_\eps(x)$ in exactly two components (see Figure \ref{fig:espinhasetor1}).
    \begin{figure}[H]
        \centering
        \includegraphics[width=0.17\textwidth]{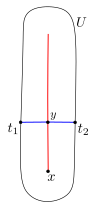}
        \caption{}
        \label{fig:espinhasetor1}
    \end{figure}
        The semi-continuity of the map $a\to C^{s}_{\frac{\eps}{2}}(a)$ (see page 15 and Theorem 6.7.1 of \cite{artigue_2018}) allows us to choose a disk $V$ centered at $x$ such that $$C^{s}_{\frac{\eps}{2}}(z)\subset U \,\,\,\,\,\, \text{for every} \,\,\,\,\,\, z\in V$$ and $C^{s}_{\frac{\eps}{2}}(x)\cap \partial V=\{\tau\}$. In particular, $\partial V\setminus \{\tau\}$ is connected. For $i=1,2$, let $c_i=u(y,t_i,y)$ and
        $$\SE_i = \{z\in \partial V\setminus \{\tau\} : C^{s}_{\frac{\eps}{2}}(z)\cap c_i \neq \emptyset\}.$$
    
        Since $\SE_i$ is closed and $\partial V\setminus \{\tau\}$ is connected, if $\SE_1$ and $\SE_2$ are not empty, then $\SE_1\cap \SE_2 \neq \emptyset$. Hence, by Lemma \ref{1or2} there is $z\in \partial V\setminus \{\tau\}$ such that $C^s_{\frac{\eps}{2}}(z)$ intersects $c_1\setminus \{y\}$ and $c_2\setminus \{y\}$. Choose intersections $a_1,a_2$, respectively, such that $s(z,a_1;z)\setminus \{a_1\}$ and $s(z,a_2;z)\setminus \{a_2\}$ are contained in the component of $x$ in $U\setminus u(t_1,t_2;y)$. Hence these arcs bound a disk $D$ with 
        $$\partial D=s(z,a_1;z)\cup s(z,a_2;z)\cup u(a_1,a_2;y).$$
        \begin{figure}[H]
        \centering
        \includegraphics[width=0.17\textwidth]{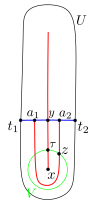}
        \caption{}
        \label{fig:espinhasetor2}
    \end{figure}
        Since $a_1,a_2\in C^{s}_{\frac{\eps}{2}}(z)$, it follows that $a_2\in C^{s}_{\eps}(a_1)$. Also, $a_1,a_2\in C^{u}_{\frac{\eps}{2}}(y)$ ensures that $a_2\in C^{u}_{\eps}(a_1)$. Thus, $$a_2\in C^{s}_{\eps}(a_1)\cap C^{u}_{\eps}(a_1).$$
        The regularity condition follows from the fact that $\inte D$ is contained in one connected component of $U\setminus u(t_1,t_2;y)$ and the intersections in $a_1$ and $a_2$ being transversal ensure that $C^s_\eps(z)$ enters the other component of $U\setminus u(t_1,t_2;y)$ through $a_1$ and $a_2$. This proves that $D$ is a regular bi-asymptotic sector and $x\in \inte D$.
        
        Now assume that $\SE_1\neq\emptyset$ and $\SE_2=\emptyset$. In this case, $\partial V\setminus \{\tau\}\subset \SE_1$. Choose $$y'\in C^s_\eps(x)\cap\inte V$$ and $u(t_1',t_2';y')$ a sub-arc of $C^u_{\frac{\eps}{2}}(y')$ such that $t_1'\not \sim t_2'$ and $u(t_1',t_2';y')$ is contained in $\inte U$ except, possibly, at $t_1'$ and $t_2'$. Since $\diam(C^u_{\frac{\eps}{2}}(y'))\geq\delta$, we can assume that either $t_1'$ or $t_2'$ belongs to $\partial U$. Let us assume that $t_1'\in\partial U$. Then we choose a sequence $(z_n)_{n\in\N}\subset\partial V\setminus \{ \tau \}$ such that 
        \begin{enumerate}
            \item $z_n\to \tau$ as $n\to \infty$,
            \item $z_n$ and $c_1$ are in different components of $$U\setminus (u(y',t_1';y') \cup s(y',y;x) \cup u(y,t_2;y))$$ for every $n\in \N$, and
            \item $C^s_{\frac{\eps}{2}}(z_n)\cap c_1 \neq \emptyset$ for every $n\in \N$.
        \end{enumerate}
    \begin{figure}[H]
        \centering
        \includegraphics[width=0.17\textwidth]{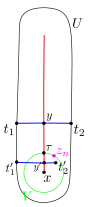}
        \caption{}
        \label{fig:espinhasetor4}
    \end{figure}
    In particular, (3) implies that $$C^s_{\frac{\eps}{2}}(z_n)\cap u(y',t_1';y')\neq\emptyset \,\,\,\,\,\, \text{for every} \,\,\,\,\,\, n\in \N.$$ This ensures the existence of $n_0\in\N$ such that $$C^s_{\frac{\eps}{2}}(z_n)\cap u(y',t_2';y')\neq\emptyset \,\,\,\,\,\, \text{for every} \,\,\,\,\,\, n\geq n_0,$$
    since the intersection between $C^{s}_{\eps}(x)$ and $u(t_1',t_2';y)$ is transversal at $y'$, and the semi-continuity again ensure that $C^s_{\frac{\eps}{2}}(z_n)$ converge to a subset of $C^s_\eps(x)$. As in the previous case, the regularity condition is assured by the transversality of these intersections and we obtain a regular bi-asymptotic sector $D$ with $x\in \inte D$.
        %Since $z_n\in \SE_1$, by item $2$ we find $a_{1n}\in C^s_{\eps/2}(z_n)\cap u(y,t_1';y)$ for every $n\in \N$. Since $z_n\to \tau$, $\tau \in C^s_{\eps/2}(\tau) \subset C^s_\eps(x)$, then $C^s_{\eps/2}(z_n)$ converges to a subset of $C^s_\eps(x)$. Since $u(t_1',t_2',y')$ intersect $C^s_\eps(x)$ transversely, we find $N>0$ such that $C^s_{\eps/2}(z_N)$ intersects $u(y,t_2';y)$ at most finitely many times. From this and since these intersections are transversal we find $a_1,a_2$ in $C^s_{\eps/2}\cap u(y,t_1';y)$ and $C^s_{\eps/2}\cap u(y,t_2';y)$ respectively such that $u(a_1,a_2;y')\cup s(a_1,a_2;z_N)$ bounds a disk $D$ with $x\in \inte D$. Since $u(a_1,a_2;y')\subset C^u_{\eps/2}(y')\subset C^u_{\eps}(z)$, then $u(a_1,a_2;y')$ is a subarc of $C^u_{\eps}(z)$. As in the previous case, the regularity condition is guaranteed by the transversality of these intersections. Therefore we have a regular bi-asymptotic sector $D$ with $x\in \inte D$.
    \begin{figure}[H]
        \centering
        \includegraphics[width=0.17\textwidth]{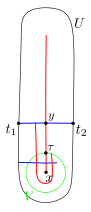}
        \caption{}
        \label{fig:espinhasetor3}
    \end{figure}   
    \end{proof}

    A direct consequence of the previous lemma and Proposition \ref{espinhasetor} is the following:
    %Using \ref{spinesector}, for all $x\in \espinha(f)$ we have $\interior D_x$ an open neighborhood of $x$ such that $D_x\cap \espinha(f)=\{x\}$. 
    
    \begin{corollary}\label{spinecountable}
        Every spine is isolated from other spines, and $\espinha(f)$ is at most countable.
    \end{corollary}

    In the proof of Lemma \ref{spinesector}, the diameter of the sector $D$ containing the spine $x$ depends on the spine and is not necessarily uniform over all spines, i.e., for a sequence of distinct spines the diameters of the respective sectors could become arbitrarily small. The following Lemma ensures the existence of bi-asymptotic sectors with uniform diameter close to the local stable/unstable continua of every spine. However, these sectors may not contain the associated spines in its interior.

    \begin{lemma}\label{setorlongo}
        Let $x\in \espinha(f)$ and $y\in C^{s}_{\frac{\eps}{2}}(x)$ with $0<\diam(s(x,y;x))<\frac{\delta}{4}$. If $u(a_1,a_2;y)$ is a sub-arc of $C^{u}_\eps(y)$ containing $y$ in its interior, then there exists a neighborhood $V$ of $x$ and $z\in \partial V$ such that $\# (C^{s}_{\eps}(z)\cap u(a_1,a_2;y))\geq 2$. %Furthermore, the intersections are transversal forming a regular bi-asymptotic sector. %Exchanging $s$ and $u$ we obtain the same result.
    \end{lemma}
    \begin{proof}%We give the proof for $\sigma=s$, for $\sigma=u$ is analogous. 
        %Let $\delta>0$ given by Remark \ref{} and let $u(a_1,a_2;y)$ be a subarc of $C^u_{\frac{\eps}{2}}(y)$ containing $y$ in its interior.
        Let $U$ be a small neighborhood of $C^s_\eps(x)$ and $u(a'_1,a'_2;y)$ be a sub-arc of $u(a_1,a_2;y)$ contained in $U$ such that 
        \begin{enumerate}
            \item $u(a'_1,a'_2;y)\cap C^s_\eps(x) = \{y\}$,
            \item $u(a'_1,a'_2;y)\cap \partial U=\{a_1',a_2'\}$,
            \item $\dist(x, u(a_1',a_2';y))<\frac{\delta}{2}$.
        \end{enumerate}
        The existence of this sub-arc is ensured by $cw_F$-expansiveness. Thus, $u(a'_1,a'_2;y)$ separates $U$ in two components and the diameter of the component $W$ of $x$ is less than $\delta$. As in the proof of Lemma \ref{spinesector}, let $V$ be an open disk centered at $x$ such that $C^s_\eps(z)\subset U$ for all $z\in V$, and let being $\tau$ be the first intersection of $\partial V$ with $C^s_\eps(x)$. There exists $r>0$ such that if $z\in\partial V\setminus\{\tau\}$ is $r$-close to $\tau$, then  $C^s_{\frac{\eps}{2}}(z)$ intersects $u(a'_1,a'_2;y)$. Since $\espinha(f)$ is at most countable (see Corollary \ref{spinecountable}), there exists $$z\in B_r(\tau)\cap\partial V\setminus\{\tau\}$$ such that $C^s_\eps(z)$ does not contain spines. Since $\diam(W)<\delta$, Lemma \ref{separadelta} ensures that $C^s_\eps(z)$ separates $W$. Since $C^s_\eps(z)\subset U$, it follows that $$\#(C^s_{\frac{\eps}{2}}(z)\cap u(a'_1,a'_2;y))\geq 2.$$ 
        %Let $w$ be one of these intersections, and we conclude that $w\in u(a_1,a_2;y)$ and
        %$$\#(C^s_{\eps}(w)\cap u(a_1,a_2;y))\geq 2.$$
        %Let $w_1\in C^s_{\frac{\eps}{2}}(z)\cap u(a'_1,a'_2;y)$.  and $C^s_{\frac{\eps}{2}}(z)$ intersects $\partial W$,  the existence of $w_2\not\sim w_1$ such that $w_2 \in C^s_{\frac{\eps}{2}}(z)\cap u(a'_1,a'_2;y)$. Therefore $C^s_{\frac{\eps}{2}}(\tau)$ intersects $u(a'_1,a'_2;y)$ twice and the claim holds.
%        \begin{affirmation1}
 %           $\#(C^s_{\frac{\eps}{2}}(\tau)\cap u(a'_1,a'_2;y))\geq 2$ for some $\tau\in \tau_{-}\tau_{+}$.
 %       \end{affirmation1}
        %\textit{Proof of Claim 1.} 
        %    \bigbreak        
        %Since $w_1\in C^s_{\frac{\eps}{2}}(\tau)$, then $C^s_{\frac{\eps}{2}}(\tau)\subset C^s_\eps(w_1)$. Therefore $C^s_\eps(w_1)$ intersects $u(a'_1,a'_2;y)$ twice.
        %The transversality of intersections is guaranteed because $u(a'_1,a'_2;y)$ does not contain spines and uses Lemma \ref{transversal}. This finishes the proof. (see Figure \ref{fig:setorlongo})
    \end{proof}

   % \begin{remark}
%         Note that the previous lemma guarantees the existence of a long bi-asymptotic sector, i.e., a uniform diameter for sectors close to the local stable continua of spines. This sector may or may not contain $x$ (See Fig.\ \ref{fig:setorlongo}). By the choice of constants $\eps$ and $\delta$ we see that $\#(C^s_\eps(w_1)\cap C^u_\eps(w_1))\geq 2$.
%    \end{remark}\textcolor{red}{Colocar o w na figura}
    \begin{figure}[H]
        \centering
        \includegraphics[width=0.60\textwidth]{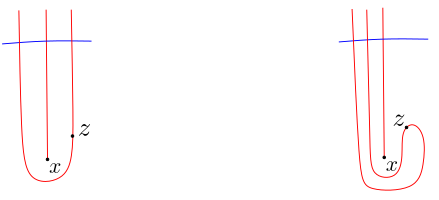}
        \caption{Long bi-asymptotic sectors close to $C^s_\eps(x)$.}
        \label{fig:setorlongo}
    \end{figure}

    \section{$Cw_F$-hyperbolicity and $cw_2$-hyperbolicity}

    In this section we prove Theorem \ref{main}. We first prove that a cw$_F$-hyperbolic surface homeomorphism that is not cw$_2$-expansive must contain infinitely many spines, and then we prove the finiteness of the set of spines for a $cw_3$-hyperbolic homeomorphism using the long bi-asymptotic sectors we constructed at the end of Section 3. These two results ensure that $cw_3$-hyperbolicity implies $cw_2$-hyperbolicity on surfaces. If a cw$_F$-hyperbolic surface homeomorphism is not cw$_1$-hyperbolic, then there exists arbitrarily small bi-asymptotic sectors, but this does not necessarily imply the existence of an infinite number of spines, since all these sectors can converge to the same single spine, as in the pseudo-Anosov diffeomorphism of $\mathbb{S}^2$. If the homeomorphism is not $cw_2$-expansive, then
    for each $\eps>0$ there exists $x\in S$ such that $\#(C^s_\eps(x)\cap C^u_\eps(x))\geq 3$.
    We note below that either there exists two bi-asymptotic sectors with disjoint interiors and connected by their boundaries (see Figure \ref{fig:2biass}) or there exists a non-regular bi-asymptotic sector. In the first case, Lemma \ref{espinhasetor} ensures that each of these sectors must contain a spine, and in the second case we will prove that inside any non-regular sector there exist at least two distinct spines. This ensures the existence of an infinite number of spines, since these sectors cannot accumulate in a pair of two spines when $\eps$ converges to zero.
    %By arbitrariness of $\eps>0$, we obtain arbitrarily small open sets containing at least two spines, which guarantees the infinity of $\espinha(f)$.

    \begin{figure}[H]
        \centering
        \includegraphics[width=0.50\textwidth]{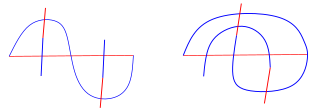}
        \caption{Examples of bi-asymptotic sectors and their spines.}
        \label{fig:2biass}
    \end{figure}
    
    \begin{proposition}\label{notregular}
    If $a^s$ and $a^u$ bound a non-regular bi-asymptotic sector $D$ with $\diam(D)\leq\delta$, then there exist at least two distinct spines in $\inte D$.
    \end{proposition}
    
    \begin{proof}
     Let $a_1$ and $a_2$ be the end points of $a^s$ and $a^u$ and assume that $$C^s_D(a_1)\cap \inte D\neq\emptyset \,\,\,\,\,\, \text{and} \,\,\,\,\,\, C^u_D(a_1)\cap \inte D\neq\emptyset.$$ Lemma \ref{separadelta} ensures that both $C^s_D(a_1)\setminus\{a^s\}$ and $C^u_D(a_1)\setminus\{a^u\}$ either separate $D$ or contain a spine in $\inte D$. If both separate $D$, then there exist two bi-asymptotic sectors with disjoint interiors and connected by their boundaries, and Lemma \ref{espinhasetor} ensures the existence of two distinct spines in $\inte D$ (see Figure \ref{fig:duasespinhas1}). 
     \begin{figure}[H]
         \centering
         \includegraphics[width=0.4\textwidth]{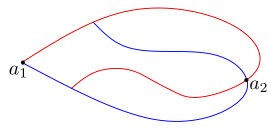}
         \caption{}
         \label{fig:duasespinhas1}
     \end{figure}     
     This figure illustrates the case where $C^s_D(a_1)\setminus\{a^s\}$ does not intersect $C^u_D(a_1)\setminus\{a^u\}$. In the case they intersect, in only a finite number of points by cw$_F$-expansive- ness, then the existence of an intersection $a_n$ such that $s(a_2,a_n;a_1)$ and $u(a_2,a_n;a_1)$ bound a regular sector, ensures the existence of two sectors with disjoint interiors and connected by their boundaries. Indeed, if $C^s_D(a_1)\setminus\{a^s\cup s(a_2,a_n;a_1)\}$ does not intersect $u(a_2,a_n;a_1)$, then $$C^s_D(a_1)\setminus\{a^s\cup s(a_2,a_n;a_1)\} \,\,\,\,\,\, \text{and} \,\,\,\,\,\, C^u_D(a_1)\setminus\{a^u\cup u(a_2,a_n;a_1)\}$$ bound a sector with interior disjoint from the first one (see Figure \ref{fig:duasespinhas22}).
     \begin{figure}[H]
         \centering
         \includegraphics[width=0.4\textwidth]{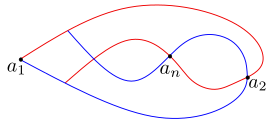}
         \caption{}
         \label{fig:duasespinhas22}
     \end{figure}
     If $C^s_D(a_1)\setminus\{a^s\cup s(a_2,a_n;a_1)\}$ intersects $u(a_2,a_n;a_1)$ at $a_m$, then it also forms a sector with disjoint interior from the first one since the regular intersection in $a_n$ ensures that $s(a_n,a_m;a_1)$ is outside the interior of the first sector. In this case, the sector bounded by $s(a_2,a_m;a_1)$ and $u(a_2,a_m;a_1)$ is not regular at $a_m$ (see Figure \ref{fig:duasespinhas4}). Since both $C^s_D(a_1)\setminus\{a^s\}$ and $C^u_D(a_1)\setminus\{a^u\}$ separate $D$, there exists an intersection $a_j$ such that $s(a_2,a_j;a_1)$ and $u(a_2,a_j;a_1)$ bound a regular sector. Indeed, if $a_m$ is a non-regular intersection as above, then $C^s_D(a_1)\setminus\{a^s\cup s(a_2,a_m;a_1)\}$ intersects $u(a_2,a_m;a_1)$ at $a_j$, and the sector formed by $s(a_2,a_j;a_1)$ and $u(a_2,a_j;a_1)$ is regular since the intersection at $a_j$ comes from inside the non-regular sector bounded by $s(a_2,a_m;a_1)$ and $u(a_2,a_m;a_1)$ (see Figure \ref{fig:duasespinhas4}).
    %The existence of an intersection $a_m$ such that the sector bounded by $s(a_2,a_m;a_1)$ and $u(a_2,a_m;a_1)$ is not regular at $a_m$ and points $$a_m^1\in s(a_2,a_m;a_1)\cap C^u_D(a_1)\setminus\{a^u\} \,\,\,\,\,\, \text{and} \,\,\,\,\,\, a_m^2\in u(a_2,a_m;a_1)\cap C^s_D(a_1)\setminus\{a^s\}$$ such that
    %the stable/unstable arcs connecting $a_m$ to $a_m^1$ and $a_m$ to $a_m^2$ intersect only at $a_m$, $a_m^1$ and $a_m^2$, then they form a pair of bi-asymptotic sectors with disjoint interiors and connected by their boundaries (see Figure \ref{fig:duasespinhas4}).
     \begin{figure}[H]
         \centering
         \includegraphics[width=0.4\textwidth]{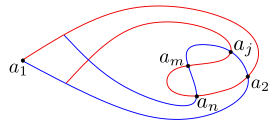}
         \caption{}
         \label{fig:duasespinhas4}
     \end{figure}
     
     %A few non-regular intersections as the above could be chained, but cw$_F$-expansive- ness ensure that only a finite number of times, and since both $C^s_D(a_1)\setminus\{a^s\}$ and $C^u_D(a_1)\setminus\{a^u\}$ separate $D$, the existence of two consecutive non-regular intersections such as $a_2$ and $a_m$ above where the stable/unstable arcs after them (in the order inside the arcs) leave the disk is assured. This ensures the existence of two sectors as in Figure \ref{fig:duasespinhas4}.
     
     %then there is $a_m$ that is the greater $a_i$ such that the sector bounded by the stable/unstable arcs connecting $a_m$ and $a_{m+1}$ is not regular in $a_{m}$ and is regular in $a_{m+1}$. , and, hence, the interior of the sector bounded by the stable/unstable arcs connecting $a_{n-1}$ and $a_{n}$ (the last intersection between $C^s_D(a_1)$ and $C^u_D(a_1)$) is disjoint from the interior of the previous sector. 
     If both $C^s_D(a_1)\setminus\{a^s\}$ and $C^u_D(a_1)\setminus\{a^u\}$ do not separate $D$, then both end in spines. If these spines are actually the same spine and $C^s_D(a_1)\setminus\{a^s\}$ and $C^u_D(a_1)\setminus\{a^u\}$ do not intersect before the spine, then they bound a bi-asymptotic sector with this spine as one of the end points of the sector, so Lemma \ref{espinhasetor} ensures the existence of a spine in the interior of this sector that is, hence, distinct from the first spine (see Figure \ref{fig:duasespinhas2}).  
     \begin{figure}[H]
         \centering
         \includegraphics[width=0.4\textwidth]{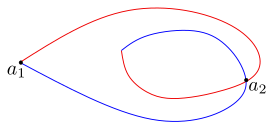}
         \caption{}
         \label{fig:duasespinhas2}
     \end{figure} 
     If $C^s_D(a_1)\setminus\{a^s\}$ and $C^u_D(a_1)\setminus\{a^u\}$ intersect before the spine in the end, then either there is an intersection bounding a regular sector and we argument as in the case above to create two sectors with disjoint interiors, or there are only non-regular intersections and we create a sector between the last one before the spine and the spine. In both cases a second spine appears. Now assume that $C^s_D(a_1)\setminus\{a^s\}$ separates $D$ but $C^u_D(a_1)\setminus\{a^u\}$ does not. If $C^s_D(a_1)\setminus\{a^s\}$ does not intersect $C^u_D(a_1)\setminus\{a^u\}$ in $\inte D$, then $C^s_D(a_1)\setminus\{a^s\}$ forms a bi-asymptotic sector with a sub-arc of $a^u$ that does not contain the spine in $C^u_D(a_1)\setminus\{a^u\}$. Then Lemma \ref{espinhasetor} ensures the existence of a spine in this sector that is, hence, distinct from the first spine (see Figure \ref{fig:setorcontido}). 
     If $C^s_D(a_1)\setminus\{a^s\}$ intersects $C^u_D(a_1)\setminus\{a^u\}$ at $a_i\in\inte D$, and $s(a_2,a_i,a_1)$ and $u(a_2,a_i,a_1)$ bound a regular sector, then there is a spine at the interior of this sector that is distinct from the spine in $C^u_D(a_1)\setminus\{a^u\}$ (see Figure \ref{fig:setorbugado21}).
     If $s(a_2,a_i,a_1)$ and $u(a_2,a_i,a_1)$ bound a non-regular sector, then the spine in $C^u_D(a_1)\setminus\{a^u\}$ is inside this sector, but since $C^s_D(a_1)\setminus\{a^s\}$ separates $D$, it must intersect $C^u_D(a_1)\setminus\{a^u\}$ an additional time creating a sector that does not contain the spine in $C^u_D(a_1)\setminus\{a^u\}$ in its interior (see Figure \ref{fig:duasespinhas44}).
     \begin{figure}[H]
         \centering
         \includegraphics[width=0.4\textwidth]{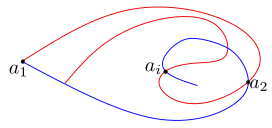}
         \caption{}
         \label{fig:duasespinhas44}
     \end{figure} 
     %this intersection forms another bi-asymptotic sector that does not contain the spine in $C^u_D(a_1)\setminus\{a^u\}$ 
     \end{proof}
     
    We are ready to prove our main theorem.
    
    \begin{proof}[Proof of Theorem \ref{main}]
    Assume that $f$ is a cw$_F$-hyperbolic homeomorphism with a finite number of spines. If $f$ is not $cw_2$-expansive, then
    for each $\alpha\in(0,\delta)$ there exists $x\in S$ such that $\#(C^s_\alpha(x)\cap C^u_\alpha(x))\geq 3$. Using an order $<$ in $C^s_\alpha(x)$ we can choose three consecutive points $a_1,a_2,a_3\in C^s_\alpha(x)\cap C^u_\alpha(x)$, that is, $a_1<a_2<a_3$ and there are no points of $C^s_\alpha(x)\cap C^u_\alpha(x)$ in $(a_1,a_2)$ and $(a_2,a_3)$. This ensures that the stable/unstable arcs connecting $a_1$ to $a_2$, and also $a_2$ to $a_3$, form bi-asymptotic sectors (this is not true in the case there are intersections in $(a_1,a_2)$ or $(a_2,a_3)$ as in Figure \ref{fig:duasespinhas4}). If the by-asymptotic sector formed by the stable/unstable arcs from $a_1$ to $a_2$ is regular in $a_2$, then the stable and unstable arcs from $a_2$ to $a_3$ form a bi-asymptotic sector with interior disjoint from the interior of the sector from $a_1$ to $a_2$. This ensures the existence of two distinct spines $\alpha$-close. If the intersection in $a_2$ is not regular, then Proposition \ref{notregular} ensures the existence of two distinct spines inside the non-regular sector. This proves that for each $\alpha\in(0,\delta)$ there exist two distinct spines $\alpha$-close, and, hence, we obtain an infinite number of distinct spines for $f$, contradicting the assumption. This proves that $f$ is cw$_2$-hyperbolic.

Now we prove that cw$_3$-hyperbolicity implies finiteness of the number of spines. This is the only step of the proof that we do not know how to prove assuming cw$_F$-hyperbolicity. Let $f$ be a $cw_3$-hyperbolic homeomorphism and assume the existence of an infinite number of distinct spines. 
For each $\alpha\in(0,\eps)$ choose $\delta_{\alpha}\in(0,\alpha)$ satisfying the Lemmas \ref{tamanho uniforme}, \ref{cwexpansivecarac}, and \ref{epsigual2eps}.
Consider $x_1$ and $x_2$ spines such that there exists $y\in C^s_\alpha(x_1)\cap C^u_\alpha(x_2)\neq \emptyset$ with 
$$\diam s(x_1,y;x_1)<\frac{\delta_\alpha}{4} \,\,\,\,\,\, \text{and} \,\,\,\,\,\, \diam u(x_2,y;x_2)<\frac{\delta_\alpha}{4}.$$
Note that $y$ is not a spine, since it is contained in the local stable continua of a spine. Lemma \ref{setorlongo} ensures the existence of long bi-asymptotic sectors close to $C^s_{\eps}(x_1)$ and $C^u_{\eps}(x_2)$ intersecting in four distinct points (see Figure \ref{fig:cw3finito4}). Since this can be done for any $\alpha>0$, it follows that $f$ is not cw$_3$-hyperbolic.
    \begin{figure}[H]
        \centering
        \includegraphics[width=0.4\textwidth]{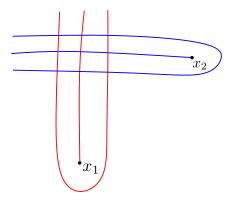}
        \caption{}
        \label{fig:cw3finito4}
    \end{figure}
    \end{proof}

\section*{Acknowledgements}
Bernardo Carvalho was supported by Progetto di Eccellenza MatMod@TOV grant number PRIN 2017S35EHN, and by CNPq grant number 405916/2018-3. Rodrigo Arruda and Alberto Sarmiento were also supported by Fapemig grant number APQ-00036-22.

    \nocite{Hiraide1987ExpansiveHO}
    \nocite{hiraide2toro}
    \nocite{Lewowicz1989ExpansiveHO}
    \nocite{Ma1979ExpansiveHA}
    \nocite{artigue2020continuumwise}
    \nocite{artigueanomalous}
    \bibliography{refs}

\begin{bibdiv}
\begin{biblist}

\bib{antunes2023firsttime}{article}{
      author={{Antunes}, Mayara},
      author={{Carvalho}, Bernardo},
       title={{First-time sensitive homeomorphisms}},
        date={2023},
     journal={arXiv e-prints},
       pages={arXiv:2304.03123},
      eprint={2304.03123},
}

\bib{artigueanomalous}{article}{
      author={Artigue, Alfonso},
       title={Anomalous cw-expansive surface homeomorphisms},
        date={2016},
        ISSN={1078-0947},
     journal={Discrete and Continuous Dynamical Systems},
      volume={36},
      number={7},
       pages={3511\ndash 3518},
         url={/article/id/5761cf90-3b5e-4230-af56-2bb485f1c121},
}

\bib{artigue_2018}{article}{
      author={Artigue, Alfonso},
       title={Dendritations of surfaces},
        date={2016},
     journal={Ergodic Theory and Dynamical Systems},
      volume={38},
       pages={2860–2912},
}

\bib{artigue2020continuumwise}{article}{
      author={{Artigue}, Alfonso},
      author={{Carvalho}, Bernardo},
      author={{Cordeiro}, Welington},
      author={{Vieitez}, Jos{\'e}},
       title={{Continuum-wise hyperbolicity}},
        date={2020-11},
     journal={arXiv e-prints},
       pages={arXiv:2011.08147},
      eprint={2011.08147},
}

\bib{ARTIGUE20203057}{article}{
      author={Artigue, Alfonso},
      author={Carvalho, Bernardo},
      author={Cordeiro, Welington},
      author={Vieitez, José},
       title={Beyond topological hyperbolicity: The l-shadowing property},
        date={2020},
        ISSN={0022-0396},
     journal={Journal of Differential Equations},
      volume={268},
      number={6},
       pages={3057\ndash 3080},
  url={https://www.sciencedirect.com/science/article/pii/S0022039619304590},
}

\bib{Artigue2013NexpansiveHO}{article}{
      author={Artigue, Alfonso},
      author={Pacifico, M.~J.},
      author={Vieitez, J.~L.},
       title={N-expansive homeomorphisms on surfaces.},
        date={2013},
     journal={Communications in Contemporary Mathematics},
      volume={19},
       pages={1650040},
}

\bib{CARVALHO20163734}{article}{
      author={Carvalho, B.},
      author={Cordeiro, W.},
       title={N-expansive homeomorphisms with the shadowing property},
        date={2016},
        ISSN={0022-0396},
     journal={Journal of Differential Equations},
      volume={261},
      number={6},
       pages={3734\ndash 3755},
  url={https://www.sciencedirect.com/science/article/pii/S0022039616301395},
}

\bib{carvalho2019positively}{article}{
      author={Carvalho, Bernardo},
      author={Cordeiro, Welington},
       title={Positively n-expansive homeomorphisms and the l-shadowing
  property},
        date={2019},
     journal={Journal of Dynamics and Differential Equations},
      volume={31},
       pages={1005\ndash 1016},
}

\bib{Artigue2019CountablyAE}{article}{
      author={Carvalho, Bernardo},
      author={Cordeiro, Welington},
      author={Artigue, Alfonso},
      author={Vieitez, Jose},
       title={Countably and entropy expansive homeomorphisms with the shadowing
  property},
        date={2022},
     journal={Proceedings of the American Mathematical Society},
      volume={150},
       pages={3369 \ndash 3378},
}

\bib{hall1955elementary}{book}{
      author={Hall, D.W.},
      author={Spencer, G.L.},
       title={Elementary topology},
   publisher={Wiley},
        date={1955},
         url={https://books.google.com.br/books?id=IANRAAAAMAAJ},
}

\bib{Hiraide1987ExpansiveHO}{article}{
      author={Hiraide, Koichi},
       title={Expansive homeomorphisms of compact surfaces are pseudo-anosov},
        date={1987},
     journal={Osaka Journal of Mathematics},
      volume={27},
       pages={117\ndash 162},
}

\bib{hiraide2toro}{article}{
      author={Hiraide, Koichi},
       title={{Expansive homeomorphisms with the pseudo-orbit tracing property
  on compact surfaces}},
        date={1988},
     journal={Journal of the Mathematical Society of Japan},
      volume={40},
      number={1},
       pages={123 \ndash  137},
         url={https://doi.org/10.2969/jmsj/04010123},
}

\bib{katoconcerning}{article}{
      author={Kato, Hisao},
       title={Concerning continuum-wise fully expansive homeomorphisms of
  continua},
        date={1993},
        ISSN={0166-8641},
     journal={Topology and its Applications},
      volume={53},
      number={3},
       pages={239\ndash 258},
  url={https://www.sciencedirect.com/science/article/pii/016686419390119X},
}

\bib{kato_1993}{article}{
      author={Kato, Hisao},
       title={Continuum-wise expansive homeomorphisms},
        date={1993},
     journal={Canadian Journal of Mathematics},
      volume={45},
      number={3},
       pages={576–598},
}

\bib{Lewowicz1989ExpansiveHO}{article}{
      author={Lewowicz, Jorge},
       title={Expansive homeomorphisms of surfaces},
        date={1989},
     journal={Boletim da Sociedade Brasileira de Matem{\'a}tica -
  Bulletin/Brazilian Mathematical Society},
      volume={20},
       pages={113\ndash 133},
}

\bib{Ma1979ExpansiveHA}{article}{
      author={Ma{\~n}{\'e}, Ricardo},
       title={Expansive homeomorphisms and topological dimension},
        date={1979},
     journal={Transactions of the American Mathematical Society},
      volume={252},
       pages={313\ndash 319},
}

\end{biblist}
\end{bibdiv}

\vspace{1.5cm}
\noindent

{\em R. Arruda}
\vspace{0.2cm}

\noindent

Departamento de Matem\'atica,

Universidade Federal de Minas Gerais - UFMG

Av. Ant\^onio Carlos, 6627 - Campus Pampulha

Belo Horizonte - MG, Brazil.
\vspace{1.0cm}

{\em B. Carvalho}
\vspace{0.2cm}

\noindent

Dipartimento di Matematica,

Università degli Studi di Roma Tor Vergata

Via Cracovia n.50 - 00133

Roma - RM, Italy

\email{mldbnr01@uniroma2.it}

\vspace{1.0cm}

{\em A. Sarmiento}
\vspace{0.2cm}

\noindent

Departamento de Matem\'atica,

Universidade Federal de Minas Gerais - UFMG

Av. Ant\^onio Carlos, 6627 - Campus Pampulha

Belo Horizonte - MG, Brazil.
\vspace{0.2cm}

\end{document}